\documentclass{amsart}
\usepackage{amsfonts}

\newtheorem{theorem}{Theorem}
\theoremstyle{plain}

\newtheorem{corollary}{Corollary}

\newtheorem{lemma}{Lemma}

\newtheorem{remark}{Remark}

\numberwithin{equation}{section}
\input{tcilatex}

\begin{document}
\title[Levin-Ste\v{c}kin's Type Inequalities for Operator Convex Functions]{%
Some Levin-Ste\v{c}kin's Type Inequalities for Operator Convex Functions on
Hilbert Spaces}
\author[S. S. Dragomir]{Silvestru Sever Dragomir$^{1,2}$}
\address{$^{1}$Mathematics, College of Engineering \& Science\\
Victoria University, PO Box 14428\\
Melbourne City, MC 8001, Australia.}
\email{sever.dragomir@vu.edu.au}
\urladdr{http://rgmia.org/dragomir}
\address{$^{2}$DST-NRF Centre of Excellence in the Mathematical \\
and Statistical Sciences, School of Computer Science\\
\& Applied Mathematics, University of the Witwatersrand, \\
Johannesburg, South Africa.}
\subjclass{47A63, 26D15, 26D10.}
\keywords{Operator convex functions, Integral inequalities, Hermite-Hadamard
inequality, F\'{e}jer's inequalities, Levin-Ste\v{c}kin's inequality.}

\begin{abstract}
Let $f$ be an operator convex function on $I$ and $A,$ $B\in \mathcal{SA}%
_{I}\left( H\right) ,$ the class of all selfadjoint operators with spectra
in $I.$ Assume that $p:\left[ 0,1\right] \rightarrow \mathbb{R}$ is
symmetric and non-decreasing on $\left[ 0,1/2\right] $. In this paper we
obtained, among others, that%
\begin{align*}
0& \leq \int_{0}^{1}p\left( t\right) dt\int_{0}^{1}f\left( tA+\left(
1-t\right) B\right) dtdt-\int_{0}^{1}p\left( t\right) f\left( tA+\left(
1-t\right) B\right) dt \\
& \leq \frac{1}{4}\left[ p\left( \frac{1}{2}\right) -p\left( 0\right) \right]
\left[ \frac{f\left( A\right) +f\left( B\right) }{2}-f\left( \frac{A+B}{2}%
\right) \right] 
\end{align*}%
in the operator order.

Several other similar inequalities for either $p$ or $f$ is differentiable,
are also provided. Applications for power function and logarithm are given
as well.
\end{abstract}

\maketitle

\section{Introduction}

A real valued continuous function $f$ on an interval $I$ is said to be 
\textit{operator convex (operator concave)} on $I$ if 
\begin{equation}
f\left( \left( 1-\lambda \right) A+\lambda B\right) \leq \left( \geq \right)
\left( 1-\lambda \right) f\left( A\right) +\lambda f\left( B\right)
\label{OC}
\end{equation}%
in the operator order, for all $\lambda \in \left[ 0,1\right] $ and for
every selfadjoint operator $A$ and $B$ on a Hilbert space $H$ whose spectra
are contained in $I.$ Notice that a function $f$ is operator concave if $-f$
is operator convex.

A real valued continuous function $f$ on an interval $I$ is said to be 
\textit{operator monotone} if it is monotone with respect to the operator
order, i.e., $A\leq B$ with $\limfunc{Sp}\left( A\right) ,$ $\limfunc{Sp}%
\left( B\right) \subset I$ imply $f\left( A\right) \leq f\left( B\right) .$

For some fundamental results on operator convex (operator concave) and
operator monotone functions, see \cite{FMPS} and the references therein.

As examples of such functions, we note that $f\left( t\right) =t^{r}$ is
operator monotone on $[0,\infty )$ if and only if $0\leq r\leq 1.$ The
function $f\left( t\right) =t^{r}$ is operator convex on $(0,\infty )$ if
either $1\leq r\leq 2$ or $-1\leq r\leq 0$ and is operator concave on $%
\left( 0,\infty \right) $ if $0\leq r\leq 1.$ The logarithmic function $%
f\left( t\right) =\ln t$ is operator monotone and operator concave on $%
(0,\infty ).$ The entropy function $f\left( t\right) =-t\ln t$ is operator
concave on $(0,\infty ).$ The exponential function $f\left( t\right) =e^{t}$
is neither operator convex nor operator monotone.

In \cite{DSSH} we obtained among others the following Hermite-Hadamard type
inequalities for operator convex functions $f:I\rightarrow \mathbb{R}$%
\begin{equation}
f\left( \frac{A+B}{2}\right) \leq \int_{0}^{1}f\left( \left( 1-s\right)
A+sB\right) ds\leq \frac{f\left( A\right) +f\left( B\right) }{2},
\label{e.1}
\end{equation}%
where $A,$ $B$ are selfadjoint operators with spectra included in $I.$

From the operator convexity of the function $f$ we have%
\begin{align}
f\left( \frac{A+B}{2}\right) & \leq \frac{1}{2}\left[ f\left( \left(
1-s\right) A+sB\right) +f\left( sA+\left( 1-s\right) B\right) \right]
\label{e.2} \\
& \leq \frac{f\left( A\right) +f\left( B\right) }{2}  \notag
\end{align}%
for all $s\in \left[ 0,1\right] $ and $A,$ $B$ selfadjoint operators with
spectra included in $I.$

If $p:\left[ 0,1\right] \rightarrow \lbrack 0,\infty )$ is Lebesgue
integrable and symmetric in the sense that $p\left( 1-s\right) =p\left(
s\right) $ for all $s\in \left[ 0,1\right] ,$ then by multiplying (\ref{e.2}%
) with $p\left( s\right) ,$ integrating on $\left[ 0,1\right] $ and taking
into account that 
\begin{equation*}
\int_{0}^{1}p\left( s\right) f\left( \left( 1-s\right) A+sB\right)
ds=\int_{0}^{1}p\left( s\right) f\left( sA+\left( 1-s\right) B\right) ds,
\end{equation*}
we get the weighted version of (\ref{e.1}) for $A,$ $B$ selfadjoint
operators with spectra included in $I$%
\begin{align}
\left( \int_{0}^{1}p\left( s\right) ds\right) f\left( \frac{A+B}{2}\right) &
\leq \int_{0}^{1}p\left( s\right) f\left( sA+\left( 1-s\right) B\right) ds
\label{e.3} \\
& \leq \left( \int_{0}^{1}p\left( s\right) ds\right) \frac{f\left( A\right)
+f\left( B\right) }{2},  \notag
\end{align}%
which are the operator version of the well known \textit{F\'{e}jer's
inequalities} for scalar convex functions.

For recent inequalities for operator convex functions see \cite{AD}-\cite{BT}%
, \cite{DDT}, \cite{DSSS}-\cite{G1}, and \cite{P}-\cite{W1}.

The following result is known in the literature as Levin-Ste\v{c}kin's
inequality \cite{LS}:

\begin{theorem}
\label{t.1}If the function $p:\left[ 0,1\right] \rightarrow \mathbb{R}$ is
symmetric, namely $p\left( 1-t\right) =p\left( t\right) $ for $t\in \left[
0,1\right] $ and non-decreasing (non-increasing) on $\left[ 0,1/2\right] ,$
then for every convex function $g$ on $\left[ 0,1\right] ,$ 
\begin{equation}
\int_{0}^{1}p\left( t\right) g\left( t\right) dt\leq \left( \geq \right)
\int_{0}^{1}p\left( t\right) dt\int_{0}^{1}g\left( t\right) dt.  \tag{LS}
\label{LS}
\end{equation}%
If the function $g$ is concave on $\left[ 0,1\right] ,$ then the signs of
inequalities reverse in (\ref{LS}).
\end{theorem}

For some recent results related to Levin-Ste\v{c}kin's inequality, see \cite%
{M1}, \cite{M2} and \cite{Wi}.

Motivated by the above operator inequalities, we provide in this paper the
operator version of Levin-Ste\v{c}kin's inequality as well as several
reverses. Applications for power function and logarithm are also given.

\section{Operator Inequalities}

Let $f$ be an operator convex function on $I.$ For $A,$ $B\in \mathcal{SA}%
_{I}\left( H\right) ,$ the class of all selfadjoint operators with spectra
in $I,$ we consider the auxiliary function $\varphi _{\left( A,B\right) }:%
\left[ 0,1\right] \rightarrow \mathcal{SA}\left( H\right) ,$ the class of
all selfadjoint operators on $H$, defined by 
\begin{equation}
\varphi _{\left( A,B\right) }\left( t\right) :=f\left( \left( 1-t\right)
A+tB\right) .  \label{e.2.1}
\end{equation}%
For $x\in H$ we can also consider the auxiliary function $\varphi _{\left(
A,B\right) ;x}:\left[ 0,1\right] \rightarrow \mathbb{R}$ defined by 
\begin{equation}
\varphi _{\left( A,B\right) ;x}\left( t\right) :=\left\langle \varphi
_{\left( A,B\right) }\left( t\right) x,x\right\rangle =\left\langle f\left(
\left( 1-t\right) A+tB\right) x,x\right\rangle .  \label{e.2.2}
\end{equation}

We have the following basic fact \cite{DSSRF}:

\begin{lemma}
\label{l.2.1}Let $f$ be an operator convex function on $I.$ For any $A,$ $%
B\in \mathcal{SA}_{I}\left( H\right) ,$ $\varphi _{\left( A,B\right) }$ is
well defined and convex in the operator order. For any $\left( A,B\right)
\in \mathcal{SA}_{I}\left( H\right) $ and $x\in H$ the function $\varphi
_{\left( A,B\right) ;x}$ is convex in the usual sense on $\left[ 0,1\right]
. $
\end{lemma}

A continuous function $g:\mathcal{SA}_{I}\left( H\right) \rightarrow 
\mathcal{B}\left( H\right) $ is said to be \textit{G\^{a}teaux differentiable%
} in $A\in \mathcal{SA}_{I}\left( H\right) $ along the direction $B\in 
\mathcal{B}\left( H\right) $ if the following limit exists in the strong
topology of $\mathcal{B}\left( H\right) $%
\begin{equation}
\nabla g_{A}\left( B\right) :=\lim_{s\rightarrow 0}\frac{g\left( A+sB\right)
-g\left( A\right) }{s}\in \mathcal{B}\left( H\right) .  \label{e.2.3}
\end{equation}%
If the limit (\ref{e.2.3}) exists for all $B\in \mathcal{B}\left( H\right) ,$
then we say that $f$ is \textit{G\^{a}teaux differentiable} in $A$ and we
can write $g\in \mathcal{G}\left( A\right) .$ If this is true for any $A$ in
an open set $\mathcal{S}$ from $\mathcal{SA}_{I}\left( H\right) $ we write
that $g\in \mathcal{G}\left( \mathcal{S}\right) .$

If $g$ is a continuous function on $I,$ by utilising the continuous
functional calculus the corresponding function of operators will be denoted
in the same way.

For two distinct operators $A,$ $B\in \mathcal{SA}_{I}\left( H\right) $ we
consider the segment of selfadjoint operators 
\begin{equation*}
\left[ A,B\right] :=\left\{ \left( 1-t\right) A+tB\text{ }|\text{ }t\in %
\left[ 0,1\right] \right\} .
\end{equation*}%
We observe that $A,$ $B\in \left[ A,B\right] $ and $\left[ A,B\right]
\subset \mathcal{SA}_{I}\left( H\right) .$

We also have \cite{DSSRF}:

\begin{lemma}
\label{l.2.2}Let $f$ be an operator convex function on $I$ and $A,$ $B\in 
\mathcal{SA}_{I}\left( H\right) ,$ with $A\neq B.$ If $f\in \mathcal{G}%
\left( \left[ A,B\right] \right) ,$ then the auxiliary function $\varphi
_{\left( A,B\right) }$ is differentiable on $\left( 0,1\right) $ and 
\begin{equation}
\varphi _{\left( A,B\right) }^{\prime }\left( t\right) =\nabla f_{\left(
1-t\right) A+tB}\left( B-A\right) .  \label{e.2.4}
\end{equation}%
Also we have for the lateral derivative that%
\begin{equation}
\varphi _{\left( A,B\right) }^{\prime }\left( 0+\right) =\nabla f_{A}\left(
B-A\right)  \label{e.2.5}
\end{equation}%
and%
\begin{equation}
\varphi _{\left( A,B\right) }^{\prime }\left( 1-\right) =\nabla f_{B}\left(
B-A\right) .  \label{e.2.6}
\end{equation}
\end{lemma}

and

\begin{lemma}
\label{l.2.3}Let $f$ be an operator convex function on $I$ and $A,$ $B\in 
\mathcal{SA}_{I}\left( H\right) ,$ with $A\neq B.$ If $f\in \mathcal{G}%
\left( \left[ A,B\right] \right) ,$ then for $0<t_{1}<t_{2}<1$ we have 
\begin{equation}
\nabla g_{\left( 1-t_{1}\right) A+t_{1}B}\left( B-A\right) \leq \nabla
g_{\left( 1-t_{2}\right) A+t_{2}B}\left( B-A\right)  \label{e.2.8}
\end{equation}%
in the operator order.

We also have 
\begin{equation}
\nabla f_{A}\left( B-A\right) \leq \nabla g_{\left( 1-t_{1}\right)
A+t_{1}B}\left( B-A\right)  \label{e.2.8.a}
\end{equation}%
and%
\begin{equation}
\nabla g_{\left( 1-t_{2}\right) A+t_{2}B}\left( B-A\right) \leq \nabla
f_{B}\left( B-A\right) .  \label{e.2.8.b}
\end{equation}
\end{lemma}

In particular, we observe that:

\begin{corollary}
\label{c.2.1}Let $f$ be an operator convex function on $I$ and $A,$ $B\in 
\mathcal{SA}_{I}\left( H\right) ,$ with $A\neq B.$ If $f\in \mathcal{G}%
\left( \left[ A,B\right] \right) $, then for all $t\in \left( 0,1\right) $
we have 
\begin{equation}
\nabla f_{A}\left( B-A\right) \leq \nabla f_{\left( 1-t\right) A+tB}\left(
B-A\right) \leq \nabla f_{B}\left( B-A\right) .  \label{e.2.10}
\end{equation}
\end{corollary}

For two \textit{Lebesgue integrable} functions $h,$ $g:\left[ a,b\right]
\rightarrow \mathbb{R}$, consider the \textit{\v{C}eby\v{s}ev functional}:%
\begin{equation}
C\left( h,g\right) :=\frac{1}{b-a}\int_{a}^{b}h(t)g(t)dt-\frac{1}{\left(
b-a\right) ^{2}}\int_{a}^{b}h(t)dt\int_{a}^{b}g(t)dt.  \label{1.1}
\end{equation}

In 1935, Gr\"{u}ss \cite{G} showed that%
\begin{equation}
\left\vert C\left( h,g\right) \right\vert \leq \frac{1}{4}\left( M-m\right)
\left( N-n\right) ,  \label{1.2}
\end{equation}%
provided that there exists the real numbers $m,$ $M,$ $n,$ $N$ such that%
\begin{equation}
m\leq h\left( t\right) \leq M\quad \text{and\quad }n\leq g\left( t\right)
\leq N\text{\quad for a.e. }t\in \left[ a,b\right] .  \label{1.3}
\end{equation}%
The constant $\frac{1}{4}$ is best possible in (\ref{1.1}) in the sense that
it cannot be replaced by a smaller quantity.

We have the following operator inequalities:

\begin{theorem}
\label{t.2.2}Let $f$ be an operator convex function on $I$ and $A,$ $B\in 
\mathcal{SA}_{I}\left( H\right) .$ Assume that $p:\left[ 0,1\right]
\rightarrow \mathbb{R}$ is symmetric and non-decreasing on $\left[ 0,1/2%
\right] $, then we have the operator inequality%
\begin{align}
0& \leq \int_{0}^{1}p\left( t\right) dt\int_{0}^{1}f\left( tA+\left(
1-t\right) B\right) dt-\int_{0}^{1}p\left( t\right) f\left( tA+\left(
1-t\right) B\right) dt  \label{e.3.1} \\
& \leq \frac{1}{4}\left[ p\left( \frac{1}{2}\right) -p\left( 0\right) \right]
\left[ \frac{f\left( A\right) +f\left( B\right) }{2}-f\left( \frac{A+B}{2}%
\right) \right] .  \notag
\end{align}%
If $p:\left[ 0,1\right] \rightarrow \mathbb{R}$ is symmetric and
non-increasing on $\left[ 0,1/2\right] $, then%
\begin{align}
0& \leq \int_{0}^{1}p\left( t\right) f\left( tA+\left( 1-t\right) B\right)
dt-\int_{0}^{1}p\left( t\right) dt\int_{0}^{1}f\left( tA+\left( 1-t\right)
B\right) dtdt  \label{e.3.1.a} \\
& \leq \frac{1}{4}\left[ p\left( 0\right) -p\left( \frac{1}{2}\right) \right]
\left[ \frac{f\left( A\right) +f\left( B\right) }{2}-f\left( \frac{A+B}{2}%
\right) \right] .  \notag
\end{align}
\end{theorem}

\begin{proof}
For $x\in H$ we consider the auxiliary function $\varphi _{\left( A,B\right)
;x}:\left[ 0,1\right] \rightarrow \mathbb{R}$ defined by 
\begin{equation*}
\varphi _{\left( A,B\right) ;x}\left( t\right) :=\left\langle \varphi
_{\left( A,B\right) }\left( t\right) x,x\right\rangle =\left\langle f\left(
\left( 1-t\right) A+tB\right) x,x\right\rangle .
\end{equation*}

Since $p$ is symmetric on $\left[ 0,1\right] ,$ then%
\begin{align*}
& \int_{0}^{1}p\left( t\right) \frac{\varphi _{\left( A,B\right) ;x}\left(
t\right) +\varphi _{\left( A,B\right) ;x}\left( 1-t\right) }{2}dt \\
& =\frac{1}{2}\left[ \int_{0}^{1}p\left( t\right) \varphi _{\left(
A,B\right) ;x}\left( t\right) dt+\int_{0}^{1}p\left( t\right) \varphi
_{\left( A,B\right) ;x}\left( 1-t\right) dt\right] \\
& =\frac{1}{2}\left[ \int_{0}^{1}p\left( t\right) \varphi _{\left(
A,B\right) ;x}\left( t\right) dt+\int_{0}^{1}p\left( 1-t\right) \varphi
_{\left( A,B\right) ;x}\left( 1-t\right) dt\right] .
\end{align*}%
By changing the variable $1-t=s,$ $s\in \left[ 0,1\right] $ we have 
\begin{equation*}
\int_{0}^{1}p\left( 1-t\right) \varphi _{\left( A,B\right) ;x}\left(
1-t\right) dt=\int_{0}^{1}p\left( s\right) \varphi _{\left( A,B\right)
;x}\left( s\right) ds
\end{equation*}%
and then 
\begin{equation*}
\int_{0}^{1}p\left( t\right) \frac{\varphi _{\left( A,B\right) ;x}\left(
t\right) +\varphi _{\left( A,B\right) ;x}\left( 1-t\right) }{2}%
dt=\int_{0}^{1}p\left( t\right) \varphi _{\left( A,B\right) ;x}\left(
t\right) dt.
\end{equation*}%
Also 
\begin{equation*}
\int_{0}^{1}\frac{\varphi _{\left( A,B\right) ;x}\left( t\right) +\varphi
_{\left( A,B\right) ;x}\left( 1-t\right) }{2}dt=\int_{0}^{1}\varphi _{\left(
A,B\right) ;x}\left( t\right) dt.
\end{equation*}%
Therefore 
\begin{align}
& \int_{0}^{1}p\left( t\right) dt\int_{0}^{1}\varphi _{\left( A,B\right)
;x}\left( t\right) dt-\int_{0}^{1}p\left( t\right) \varphi _{\left(
A,B\right) ;x}\left( t\right) dt  \label{e.3.2} \\
& =\int_{0}^{1}p\left( t\right) dt\int_{0}^{1}\check{\varphi}_{\left(
A,B\right) ;x}\left( t\right) dt-\int_{0}^{1}p\left( t\right) \check{\varphi}%
_{\left( A,B\right) ;x}\left( t\right) dt,  \notag
\end{align}%
where 
\begin{equation*}
\check{\varphi}_{\left( A,B\right) ;x}\left( t\right) :=\frac{\varphi
_{\left( A,B\right) ;x}\left( t\right) +\varphi _{\left( A,B\right)
;x}\left( 1-t\right) }{2},\text{ }t\in \left[ 0,1\right]
\end{equation*}%
is the symmetrical transform of $\varphi _{\left( A,B\right) ;x}$ on the
interval $\left[ 0,1\right] .$

Now, if we use the Levin-Ste\v{c}kin's inequality for the symmetric function 
$p$ and the convex function $g=\varphi _{\left( A,B\right) ;x}$, then we
obtain%
\begin{equation}
0\leq \int_{0}^{1}p\left( t\right) dt\int_{0}^{1}\varphi _{\left( A,B\right)
;x}\left( t\right) dt-\int_{0}^{1}p\left( t\right) \varphi _{\left(
A,B\right) ;x}\left( t\right) dt,  \label{e.3.3}
\end{equation}%
for all $x\in H$.

Since, by Lemma \ref{l.2.1}, $\varphi _{\left( A,B\right) ;x}$ is convex,
then $\check{\varphi}_{\left( A,B\right) ;x}$ is symmetric and convex, which
implies that 
\begin{align*}
\varphi _{\left( A,B\right) ;x}\left( \frac{1}{2}\right) & =\check{\varphi}%
_{\left( A,B\right) ;x}\left( \frac{1}{2}\right) \leq \check{\varphi}%
_{\left( A,B\right) ;x}\left( t\right) \\
& \leq \check{\varphi}_{\left( A,B\right) ;x}\left( 1\right) =\frac{\varphi
_{\left( A,B\right) ;x}\left( 0\right) +\varphi _{\left( A,B\right)
;x}\left( 1\right) }{2},\text{ }t\in \left[ 0,1\right] ,
\end{align*}%
for all $x\in H$.

Also $p\left( 0\right) \leq p\left( t\right) \leq p\left( \frac{1}{2}\right)
,$ $t\in \left[ 0,1\right] $ and by Gr\"{u}ss' inequality for $h=p$ and $g=%
\check{\varphi}_{\left( A,B\right) ;x}$ we get 
\begin{align*}
0& \leq \int_{0}^{1}p\left( t\right) dt\int_{0}^{1}\check{\varphi}_{\left(
A,B\right) ;x}\left( t\right) dt-\int_{0}^{1}p\left( t\right) \check{\varphi}%
_{\left( A,B\right) ;x}\left( t\right) dt \\
& \leq \frac{1}{4}\left[ p\left( \frac{1}{2}\right) -p\left( 0\right) \right]
\left[ \frac{\varphi _{\left( A,B\right) ;x}\left( 0\right) +\varphi
_{\left( A,B\right) ;x}\left( 1\right) }{2}-\varphi _{\left( A,B\right)
;x}\left( \frac{1}{2}\right) \right]
\end{align*}%
namely, by (\ref{e.3.2}) and (\ref{e.3.3}) 
\begin{align}
0& \leq \int_{0}^{1}p\left( t\right) dt\int_{0}^{1}\varphi _{\left(
A,B\right) ;x}\left( t\right) dt-\int_{0}^{1}p\left( t\right) \varphi
_{\left( A,B\right) ;x}\left( t\right) dt  \label{e.3.4} \\
& \leq \frac{1}{4}\left[ p\left( \frac{1}{2}\right) -p\left( 0\right) \right]
\left[ \frac{\varphi _{\left( A,B\right) ;x}\left( 0\right) +\varphi
_{\left( A,B\right) ;x}\left( 1\right) }{2}-\varphi _{\left( A,B\right)
;x}\left( \frac{1}{2}\right) \right]  \notag
\end{align}%
for all $x\in H$.

The inequality (\ref{e.3.4}) can be written in terms of inner product as%
\begin{align*}
0& \leq \left\langle \left( \int_{0}^{1}p\left( t\right)
dt\int_{0}^{1}f\left( \left( 1-t\right) A+tB\right) dt\right)
x,x\right\rangle \\
& -\left\langle \left( \int_{0}^{1}p\left( t\right) f\left( \left(
1-t\right) A+tB\right) dt\right) x,x\right\rangle \\
& \leq \frac{1}{4}\left[ p\left( \frac{1}{2}\right) -p\left( 0\right) \right]
\left[ \left\langle \left( \frac{f\left( A\right) +f\left( B\right) }{2}%
\right) x,x\right\rangle -\left\langle f\left( \frac{A+B}{2}\right)
x,x\right\rangle \right]
\end{align*}%
for all $x\in H$, which is equivalent to the operator inequality (\ref{e.3.1}%
).
\end{proof}

\begin{remark}
\label{r.2.1}If $f$ is an operator concave function on $I$ and $A,$ $B\in 
\mathcal{SA}_{I}\left( H\right) $, while $p:\left[ 0,1\right] \rightarrow 
\mathbb{R}$ is symmetric and non-decreasing on $\left[ 0,1/2\right] $, then 
\begin{align}
0& \leq \int_{0}^{1}p\left( t\right) f\left( tA+\left( 1-t\right) B\right)
dt-\int_{0}^{1}p\left( t\right) dt\int_{0}^{1}f\left( tA+\left( 1-t\right)
B\right) dtdt  \label{e.3.5} \\
& \leq \frac{1}{4}\left[ p\left( \frac{1}{2}\right) -p\left( 0\right) \right]
\left[ f\left( \frac{A+B}{2}\right) -\frac{f\left( A\right) +f\left(
B\right) }{2}\right] .  \notag
\end{align}%
Also, in this case, if $p:\left[ 0,1\right] \rightarrow \mathbb{R}$ is
symmetric and non-increasing on $\left[ 0,1/2\right] $, then%
\begin{align}
0& \leq \int_{0}^{1}p\left( t\right) f\left( tA+\left( 1-t\right) B\right)
dt-\int_{0}^{1}p\left( t\right) dt\int_{0}^{1}f\left( tA+\left( 1-t\right)
B\right) dtdt  \label{e.3.6} \\
& \leq \frac{1}{4}\left[ p\left( 0\right) -p\left( \frac{1}{2}\right) \right]
\left[ f\left( \frac{A+B}{2}\right) -\frac{f\left( A\right) +f\left(
B\right) }{2}\right] .  \notag
\end{align}
\end{remark}

The following inequality obtained by Ostrowski in 1970, \cite{O} also holds%
\begin{equation}
\left\vert C\left( h,g\right) \right\vert \leq \frac{1}{8}\left( b-a\right)
\left( M-m\right) \left\Vert g^{\prime }\right\Vert _{\infty },  \label{1.5}
\end{equation}%
provided that $h$ is \textit{Lebesgue integrable} and satisfies (\ref{1.3})
while $g$ is absolutely continuous and $g^{\prime }\in L_{\infty }\left[ a,b%
\right] .$ The constant $\frac{1}{8}$ is best possible in (\ref{1.5}).

We have the following operator inequalities when some differentiability
conditions are imposed.

\begin{theorem}
\label{t.2.3}Let $f$ be an operator convex function on $I$ and $A,$ $B\in 
\mathcal{SA}_{I}\left( H\right) $ while $p:\left[ 0,1\right] \rightarrow 
\mathbb{R}$ is symmetric and non-decreasing on $\left[ 0,1/2\right] $.

\begin{enumerate}
\item[(i)] If $p$ is differentiable on $\left( 0,1\right) $, then%
\begin{align}
0& \leq \int_{0}^{1}p\left( t\right) dt\int_{0}^{1}f\left( tA+\left(
1-t\right) B\right) dt-\int_{0}^{1}p\left( t\right) f\left( tA+\left(
1-t\right) B\right) dt  \label{e.3.7} \\
& \leq \frac{1}{8}\left\Vert p^{\prime }\right\Vert _{\infty }\left[ \frac{%
f\left( A\right) +f\left( B\right) }{2}-f\left( \frac{A+B}{2}\right) \right]
.  \notag
\end{align}

\item[(ii)] If $f\in \mathcal{G}\left( \left[ A,B\right] \right) ,$ then%
\begin{align}
0& \leq \int_{0}^{1}p\left( t\right) dt\int_{0}^{1}f\left( tA+\left(
1-t\right) B\right) dt-\int_{0}^{1}p\left( t\right) f\left( tA+\left(
1-t\right) B\right) dt  \label{e.3.8} \\
& \leq \frac{1}{16}\left[ p\left( \frac{1}{2}\right) -p\left( 0\right) %
\right] \left[ \nabla f_{B}\left( B-A\right) -\nabla f_{A}\left( B-A\right) %
\right] .  \notag
\end{align}
\end{enumerate}
\end{theorem}

\begin{proof}
The inequality (\ref{e.3.7}) follows by (\ref{1.5}) for $g=p$ and $h=\varphi
_{\left( A,B\right) ;x},$ $x\in H$ and proceed like in the proof of Theorem %
\ref{t.2.2}.

Now, by Lemma \ref{l.2.2}%
\begin{align}
& \left( \check{\varphi}_{\left( A,B\right) ;x}\left( t\right) \right)
^{\prime }  \label{e.3.8.a} \\
& =\frac{\left( \varphi _{\left( A,B\right) ;x}\left( t\right) \right)
^{\prime }+\left( \varphi _{\left( A,B\right) ;x}\left( 1-t\right) \right)
^{\prime }}{2}  \notag \\
& =\frac{\left\langle \varphi _{\left( A,B\right) }^{\prime }\left( t\right)
x,x\right\rangle -\left\langle \varphi _{\left( A,B\right) }^{\prime }\left(
1-t\right) x,x\right\rangle }{2}  \notag \\
& =\frac{\left\langle \nabla f_{\left( 1-t\right) A+tB}\left( B-A\right)
x,x\right\rangle -\left\langle \nabla f_{tA+\left( 1-t\right) B}\left(
B-A\right) x,x\right\rangle }{2}  \notag \\
& =\frac{\left\langle \left[ \nabla f_{\left( 1-t\right) A+tB}\left(
B-A\right) -\nabla f_{tA+\left( 1-t\right) B}\left( B-A\right) \right]
x,x\right\rangle }{2}  \notag \\
& =\left\langle \left[ \frac{\nabla f_{\left( 1-t\right) A+tB}\left(
B-A\right) -\nabla f_{tA+\left( 1-t\right) B}\left( B-A\right) }{2}\right]
x,x\right\rangle  \notag
\end{align}%
for all $t\in \left( 0,1\right) $ and any $x\in H$.

Since $\check{\varphi}_{\left( A,B\right) ;x}$ is convex on $\left(
0,1\right) ,$ then 
\begin{equation*}
\left( \check{\varphi}_{\left( A,B\right) ;x}\left( t\right) \right)
_{t=0+}^{\prime }\leq \left( \check{\varphi}_{\left( A,B\right) ;x}\left(
t\right) \right) ^{\prime }\leq \left( \check{\varphi}_{\left( A,B\right)
;x}\left( t\right) \right) _{t=1-}^{\prime },\text{ }t\in \left( 0,1\right)
\end{equation*}%
namely, by Lemma \ref{l.2.3}%
\begin{align*}
& \left\langle \left[ \frac{\nabla f_{A}\left( B-A\right) -\nabla
f_{B}\left( B-A\right) }{2}\right] x,x\right\rangle \\
& \leq \left\langle \left[ \frac{\nabla f_{\left( 1-t\right) A+tB}\left(
B-A\right) -\nabla f_{tA+\left( 1-t\right) B}\left( B-A\right) }{2}\right]
x,x\right\rangle \\
& \leq \left\langle \left[ \frac{\nabla f_{B}\left( B-A\right) -\nabla
f_{A}\left( B-A\right) }{2}\right] x,x\right\rangle
\end{align*}%
for all $t\in \left( 0,1\right) $ and any $x\in H$.

Therefore 
\begin{equation*}
\left\vert \left( \check{\varphi}_{\left( A,B\right) ;x}\left( t\right)
\right) ^{\prime }\right\vert \leq \left\vert \left\langle \left[ \frac{%
\nabla f_{B}\left( B-A\right) -\nabla f_{A}\left( B-A\right) }{2}\right]
x,x\right\rangle \right\vert
\end{equation*}%
for all $t\in \left( 0,1\right) $ and any $x\in H$, which implies that 
\begin{align*}
\sup_{t\in \left( 0,1\right) }\left\vert \left( \check{\varphi}_{\left(
A,B\right) ;x}\left( t\right) \right) ^{\prime }\right\vert & \leq
\left\vert \left\langle \left[ \frac{\nabla f_{B}\left( B-A\right) -\nabla
f_{A}\left( B-A\right) }{2}\right] x,x\right\rangle \right\vert \\
& =\left\langle \left[ \frac{\nabla f_{B}\left( B-A\right) -\nabla
f_{A}\left( B-A\right) }{2}\right] x,x\right\rangle
\end{align*}%
for any $x\in H,$ since by Corollary \ref{c.2.1}, we have $f_{B}\left(
B-A\right) \geq \nabla f_{A}\left( B-A\right) .$

If we use Ostrowski's inequality (\ref{1.5}) for $h=p$ and $g=\check{\varphi}%
_{\left( A,B\right) ;x},$ then we obtain%
\begin{align*}
0& \leq \int_{0}^{1}p\left( t\right) dt\int_{0}^{1}\check{\varphi}_{\left(
A,B\right) ;x}\left( t\right) dt-\int_{0}^{1}p\left( t\right) \check{\varphi}%
_{\left( A,B\right) ;x}\left( t\right) dt \\
& \leq \frac{1}{8}\left[ p\left( \frac{1}{2}\right) -p\left( 0\right) \right]
\sup_{t\in \left( 0,1\right) }\left\vert \left( \check{\varphi}_{\left(
A,B\right) ;x}\left( t\right) \right) ^{\prime }\right\vert \\
& \leq \frac{1}{8}\left[ p\left( \frac{1}{2}\right) -p\left( 0\right) \right]
\left\langle \left[ \frac{\nabla f_{B}\left( B-A\right) -\nabla f_{A}\left(
B-A\right) }{2}\right] x,x\right\rangle ,
\end{align*}%
namely%
\begin{align*}
0& \leq \left\langle \left( \int_{0}^{1}p\left( t\right)
dt\int_{0}^{1}f\left( \left( 1-t\right) A+tB\right) dt-\int_{0}^{1}p\left(
t\right) f\left( \left( 1-t\right) A+tB\right) dt\right) x,x\right\rangle \\
& \leq \frac{1}{8}\left[ p\left( \frac{1}{2}\right) -p\left( 0\right) \right]
\left\langle \left[ \frac{\nabla f_{B}\left( B-A\right) -\nabla f_{A}\left(
B-A\right) }{2}\right] x,x\right\rangle ,
\end{align*}%
which is equivalent to the operator inequality (\ref{e.3.8}).
\end{proof}

Another, however less known result, even though it was obtained by \v{C}eby%
\v{s}ev in 1882, \cite{C}, states that%
\begin{equation}
\left\vert C\left( h,g\right) \right\vert \leq \frac{1}{12}\left\Vert
h^{\prime }\right\Vert _{\infty }\left\Vert g^{\prime }\right\Vert _{\infty
}\left( b-a\right) ^{2},  \label{1.4}
\end{equation}%
provided that $h^{\prime },$ $g^{\prime }$ exist and are continuous on $%
\left[ a,b\right] $ and $\left\Vert h^{\prime }\right\Vert _{\infty
}=\sup_{t\in \lbrack a,b]}\left\vert h^{\prime }\left( t\right) \right\vert
. $ The constant $\frac{1}{12}$ cannot be improved in the general case.

The case of \textit{euclidean norms} of the derivative was considered by A.
Lupa\c{s} in \cite{Lu} in which he proved that%
\begin{equation}
\left\vert C\left( h,g\right) \right\vert \leq \frac{1}{\pi ^{2}}\left\Vert
h^{\prime }\right\Vert _{2}\left\Vert g^{\prime }\right\Vert _{2}\left(
b-a\right) ,  \label{1.6}
\end{equation}%
provided that $h,$ $g$ are absolutely continuous and $h^{\prime },$ $%
g^{\prime }\in L_{2}\left[ a,b\right] .$ The constant $\frac{1}{\pi ^{2}}$
is the best possible.

Further, we have:

\begin{theorem}
\label{t.2.4}Let $f$ be an operator convex function on $I$ and $A,$ $B\in 
\mathcal{SA}_{I}\left( H\right) $ while $p:\left[ 0,1\right] \rightarrow 
\mathbb{R}$ is symmetric and non-decreasing on $\left[ 0,1/2\right] $.

\begin{enumerate}
\item[(i)] If $p$ is differentiable on $\left( 0,1\right) $ and $f\in 
\mathcal{G}\left( \left[ A,B\right] \right) ,$ then%
\begin{align}
0& \leq \int_{0}^{1}p\left( t\right) dt\int_{0}^{1}f\left( tA+\left(
1-t\right) B\right) dt-\int_{0}^{1}p\left( t\right) f\left( tA+\left(
1-t\right) B\right) dt  \label{e.3.9} \\
& \leq \frac{1}{24}\left\Vert p^{\prime }\right\Vert _{\infty }\left[ \nabla
f_{B}\left( B-A\right) -\nabla f_{A}\left( B-A\right) \right] .  \notag
\end{align}

\item[(ii)] If $p$ is differentiable on $\left( 0,1\right) $ with $p^{\prime
}\in L_{2}\left[ 0,1\right] $ and $f\in \mathcal{G}\left( \left[ A,B\right]
\right) ,$ then%
\begin{align}
0& \leq \int_{0}^{1}p\left( t\right) dt\int_{0}^{1}f\left( tA+\left(
1-t\right) B\right) dt-\int_{0}^{1}p\left( t\right) f\left( tA+\left(
1-t\right) B\right) dt  \label{e.3.9.a} \\
& \leq \frac{1}{2\pi ^{2}}\left\Vert p^{\prime }\right\Vert _{2}  \notag \\
& \times \left( \int_{0}^{1}\left\Vert \nabla f_{\left( 1-t\right)
A+tB}\left( B-A\right) -\nabla f_{tA+\left( 1-t\right) B}\left( B-A\right)
\right\Vert ^{2}dt\right) ^{1/2}1_{H}  \notag \\
& \leq \frac{1}{\pi ^{2}}\left\Vert p^{\prime }\right\Vert _{2}\left(
\int_{0}^{1}\left\Vert \nabla f_{\left( 1-t\right) A+tB}\left( B-A\right)
\right\Vert ^{2}dt\right) ^{1/2}1_{H}  \notag
\end{align}%
provided the last integral is finite.
\end{enumerate}
\end{theorem}

\begin{proof}
The inequality (\ref{e.3.9}) follows by (\ref{1.4}) for $h=p$ and $g=\varphi
_{\left( A,B\right) ;x},$ $x\in H$ and proceed like in the proof of Theorem %
\ref{t.2.2}.

From (\ref{e.3.8.a}) we have%
\begin{align}
& \int_{0}^{1}\left[ \left( \check{\varphi}_{\left( A,B\right) ;x}\left(
t\right) \right) ^{\prime }\right] ^{2}dt  \label{e.3.10} \\
& =\int_{0}^{1}\left\vert \left\langle \left[ \frac{\nabla f_{\left(
1-t\right) A+tB}\left( B-A\right) -\nabla f_{tA+\left( 1-t\right) B}\left(
B-A\right) }{2}\right] x,x\right\rangle \right\vert ^{2}dt  \notag \\
& \leq \int_{0}^{1}\left\Vert \left[ \frac{\nabla f_{\left( 1-t\right)
A+tB}\left( B-A\right) -\nabla f_{tA+\left( 1-t\right) B}\left( B-A\right) }{%
2}\right] x\right\Vert ^{2}\left\Vert x\right\Vert ^{2}dt  \notag \\
& \leq \frac{1}{4}\left\Vert x\right\Vert ^{4}\int_{0}^{1}\left\Vert \nabla
f_{\left( 1-t\right) A+tB}\left( B-A\right) -\nabla f_{tA+\left( 1-t\right)
B}\left( B-A\right) \right\Vert ^{2}dt  \notag
\end{align}%
for all $x\in H,$ implying that%
\begin{align}
& \left( \int_{0}^{1}\left[ \left( \check{\varphi}_{\left( A,B\right)
;x}\left( t\right) \right) ^{\prime }\right] ^{2}dt\right) ^{1/2}
\label{e.3.11} \\
& \leq \frac{1}{2}\left\Vert x\right\Vert ^{2}\left( \int_{0}^{1}\left\Vert
\nabla f_{\left( 1-t\right) A+tB}\left( B-A\right) -\nabla f_{tA+\left(
1-t\right) B}\left( B-A\right) \right\Vert ^{2}dt\right) ^{1/2}  \notag
\end{align}%
for all $x\in H.$

By using (\ref{1.6}) for $h=p$ and $g=\varphi _{\left( A,B\right) ;x},$ $%
x\in H,$ we derive%
\begin{align*}
0& \leq \left\langle \left( \int_{0}^{1}p\left( t\right)
dt\int_{0}^{1}f\left( \left( 1-t\right) A+tB\right) dt-\int_{0}^{1}p\left(
t\right) f\left( \left( 1-t\right) A+tB\right) dt\right) x,x\right\rangle  \\
& \leq \frac{1}{2\pi ^{2}}\left\Vert p^{\prime }\right\Vert _{2}\left\Vert
x\right\Vert ^{2}\left( \int_{0}^{1}\left\Vert \nabla f_{\left( 1-t\right)
A+tB}\left( B-A\right) -\nabla f_{tA+\left( 1-t\right) B}\left( B-A\right)
\right\Vert ^{2}dt\right) ^{1/2},
\end{align*}%
which is equivalent to the first inequality in (\ref{e.3.9.a}).

By the triangle inequality, we have 
\begin{align*}
& \left( \int_{0}^{1}\left\Vert \nabla f_{\left( 1-t\right) A+tB}\left(
B-A\right) -\nabla f_{tA+\left( 1-t\right) B}\left( B-A\right) \right\Vert
^{2}dt\right) ^{1/2} \\
& \leq \left( \int_{0}^{1}\left\Vert \nabla f_{\left( 1-t\right) A+tB}\left(
B-A\right) \right\Vert ^{2}dt\right) ^{1/2}+\left( \int_{0}^{1}\left\Vert
\nabla f_{tA+\left( 1-t\right) B}\left( B-A\right) \right\Vert ^{2}dt\right)
^{1/2} \\
& =2\left( \int_{0}^{1}\left\Vert \nabla f_{\left( 1-t\right) A+tB}\left(
B-A\right) \right\Vert ^{2}dt\right) ^{1/2},
\end{align*}%
which proves the last part of (\ref{e.3.9.a}).
\end{proof}

\begin{remark}
\label{r.2.2}If either $p$ is non-increasing on $\left[ 0,1/2\right] $ or $f$
is an operator concave function on $I,$ then the interested reader may state
similar results to the ones in Theorem \ref{t.2.3} and Theorem \ref{t.2.4}.
We omit the details.
\end{remark}

\section{Some Examples}

The function $f\left( t\right) =t^{r}$ is operator convex on $(0,\infty )$
if either $1\leq r\leq 2$ or $-1\leq r\leq 0$. Assume that $p:\left[ 0,1%
\right] \rightarrow \mathbb{R}$ is symmetric and non-decreasing on $\left[
0,1/2\right] $, then we have by (\ref{e.3.1}) the operator inequality%
\begin{align}
0& \leq \int_{0}^{1}p\left( t\right) dt\int_{0}^{1}\left( tA+\left(
1-t\right) B\right) ^{r}dt-\int_{0}^{1}p\left( t\right) \left( tA+\left(
1-t\right) B\right) ^{r}dt  \label{e.4.1} \\
& \leq \frac{1}{4}\left[ p\left( \frac{1}{2}\right) -p\left( 0\right) \right]
\left[ \frac{A^{r}+B^{r}}{2}-\left( \frac{A+B}{2}\right) ^{r}\right]   \notag
\end{align}%
for all $A,$ $B>0.$

Moreover, if $p$ is differentiable on $\left( 0,1\right) $, then by (\ref%
{e.3.7})%
\begin{align}
0& \leq \int_{0}^{1}p\left( t\right) dt\int_{0}^{1}\left( tA+\left(
1-t\right) B\right) ^{r}dt-\int_{0}^{1}p\left( t\right) \left( tA+\left(
1-t\right) B\right) ^{r}dt  \label{e.4.2} \\
& \leq \frac{1}{8}\left\Vert p^{\prime }\right\Vert _{\infty }\left[ \frac{%
A^{r}+B^{r}}{2}-\left( \frac{A+B}{2}\right) ^{r}\right]   \notag
\end{align}%
for all $A,$ $B>0.$

The function $f\left( x\right) =x^{-1}$ is  operator convex on $\left(
0,\infty \right) $, operator G\^{a}teaux differentiable and%
\begin{equation*}
\nabla f_{T}\left( S\right) =-T^{-1}ST^{-1}
\end{equation*}%
for $T,$ $S>0.$

If we use (\ref{e.3.8}), then we get the inequality 
\begin{align}
0& \leq \int_{0}^{1}p\left( t\right) dt\int_{0}^{1}\left( tA+\left(
1-t\right) B\right) ^{-1}dt-\int_{0}^{1}p\left( t\right) \left( tA+\left(
1-t\right) B\right) ^{-1}dt  \label{e.4.2.a} \\
& \leq \frac{1}{16}\left[ p\left( \frac{1}{2}\right) -p\left( 0\right) %
\right] \left[ A^{-1}\left( B-A\right) A^{-1}-B^{-1}\left( B-A\right) B^{-1}%
\right]   \notag
\end{align}%
provided that $p:\left[ 0,1\right] \rightarrow \mathbb{R}$ is symmetric and
non-decreasing on $\left[ 0,1/2\right] $ and $A,$ $B>0.$

Moreover, if $p$ is differentiable on $\left( 0,1\right) $ then by (\ref%
{e.3.9}) we derive%
\begin{align}
0& \leq \int_{0}^{1}p\left( t\right) dt\int_{0}^{1}\left( tA+\left(
1-t\right) B\right) ^{-1}dt-\int_{0}^{1}p\left( t\right) \left( tA+\left(
1-t\right) B\right) ^{-1}d  \label{e.4.2.b} \\
& \leq \frac{1}{24}\left\Vert p^{\prime }\right\Vert _{\infty }\left[
A^{-1}\left( B-A\right) A^{-1}-B^{-1}\left( B-A\right) B^{-1}\right]   \notag
\end{align}%
for $A,$ $B>0.$

If we use the first and last term in (\ref{e.3.9.a}), then we also have 
\begin{align}
0& \leq \int_{0}^{1}p\left( t\right) dt\int_{0}^{1}\left( tA+\left(
1-t\right) B\right) ^{-1}dt-\int_{0}^{1}p\left( t\right) \left( tA+\left(
1-t\right) B\right) ^{-1}d  \label{e.4.2.c} \\
& \leq \frac{1}{\pi ^{2}}\left\Vert p^{\prime }\right\Vert _{2}  \notag \\
& \times \left( \int_{0}^{1}\left\Vert \left( \left( 1-t\right) A+tB\right)
^{-1}\left( B-A\right) \left( \left( 1-t\right) A+tB\right) ^{-1}\right\Vert
^{2}dt\right) ^{1/2}1_{H},  \notag
\end{align}%
provided that $p^{\prime }\in L_{2}\left[ 0,1\right] $ and $A,$ $B>0.$

The logarithmic function $f\left( t\right) =\ln t$ is operator concave on $%
(0,\infty ).$ Assume that $p:\left[ 0,1\right] \rightarrow \mathbb{R}$ is
symmetric and non-decreasing on $\left[ 0,1/2\right] $, then we have by (\ref%
{e.3.1}) the operator inequality%
\begin{align}
0& \leq \int_{0}^{1}p\left( t\right) \ln \left( tA+\left( 1-t\right)
B\right) dt-\int_{0}^{1}p\left( t\right) dt\int_{0}^{1}\ln \left( tA+\left(
1-t\right) B\right) dt  \label{e.4.3} \\
& \leq \frac{1}{4}\left[ p\left( \frac{1}{2}\right) -p\left( 0\right) \right]
\left[ \ln \left( \frac{A+B}{2}\right) -\frac{\ln A+\ln B}{2}\right] , 
\notag
\end{align}%
for all $A,$ $B>0.$

Moreover, if $p$ is differentiable on $\left( 0,1\right) $, then by (\ref%
{e.3.7})%
\begin{align}
0& \leq \int_{0}^{1}p\left( t\right) \ln \left( tA+\left( 1-t\right)
B\right) dt-\int_{0}^{1}p\left( t\right) dt\int_{0}^{1}\ln \left( tA+\left(
1-t\right) B\right) dt  \label{e.4.4} \\
& \leq \frac{1}{8}\left\Vert p^{\prime }\right\Vert _{\infty }\left[ \ln
\left( \frac{A+B}{2}\right) -\frac{\ln A+\ln B}{2}\right] ,  \notag
\end{align}%
for all $A,$ $B>0.$

We note that the function $f(x)=\ln x$ is operator concave on $\left(
0,\infty \right) .$ The $\ln $ function is operator G\^{a}teaux
differentiable with the following explicit formula for the derivative (cf.
Pedersen \cite[p. 155]{P}):%
\begin{equation}
\nabla \ln _{T}\left( S\right) =\int_{0}^{\infty }\left( s1_{H}+T\right)
^{-1}S\left( s1_{H}+T\right) ^{-1}ds  \label{e.4.5}
\end{equation}%
for $T,$ $S>0.$

If we use inequality (\ref{e.3.8}) for $\ln $ we get for $A,$ $B>0,$ 
\begin{align}
0& \leq \int_{0}^{1}p\left( t\right) \ln \left( tA+\left( 1-t\right)
B\right) dt-\int_{0}^{1}p\left( t\right) dt\int_{0}^{1}\ln \left( tA+\left(
1-t\right) B\right) dt  \label{e.4.6} \\
& \leq \frac{1}{16}\left[ p\left( \frac{1}{2}\right) -p\left( 0\right) %
\right] \left[ \int_{0}^{\infty }\left( s1_{H}+B\right) ^{-1}\left(
B-A\right) \left( s1_{H}+B\right) ^{-1}ds\right.  \notag \\
& -\left. \int_{0}^{\infty }\left( s1_{H}+A\right) ^{-1}\left( B-A\right)
\left( s1_{H}+A\right) ^{-1}ds\right] ,  \notag
\end{align}%
provided that $p:\left[ 0,1\right] \rightarrow \mathbb{R}$ is symmetric and
non-decreasing on $\left[ 0,1/2\right] .$

If $p$ is differentiable, then by (\ref{e.3.9}) we 
\begin{align}
0& \leq \int_{0}^{1}p\left( t\right) \ln \left( tA+\left( 1-t\right)
B\right) dt-\int_{0}^{1}p\left( t\right) dt\int_{0}^{1}\ln \left( tA+\left(
1-t\right) B\right) dt  \label{e.4.7} \\
& \leq \frac{1}{24}\left\Vert p^{\prime }\right\Vert _{\infty }\left[
\int_{0}^{\infty }\left( s1_{H}+B\right) ^{-1}\left( B-A\right) \left(
s1_{H}+B\right) ^{-1}ds\right.   \notag \\
& -\left. \int_{0}^{\infty }\left( s1_{H}+A\right) ^{-1}\left( B-A\right)
\left( s1_{H}+A\right) ^{-1}ds\right] ,  \notag
\end{align}%
for $A,$ $B>0.$

A similar inequality can be derive from (\ref{e.3.9.a}), however the details
are omitted.

The interested author can also state the corresponding operator inequalities
for $f\left( t\right) =t\ln t$ that is operator convex on $(0,\infty ).$

Finally, if we take $p\left( t\right) =t\left( 1-t\right) ,$ $t\in \left[ 0,1%
\right] ,$ then we observe that $p$ is symmetric and non-decreasing on $%
\left[ 0,1/2\right] $ and by (\ref{e.4.1}) we obtain   
\begin{align}
0& \leq \frac{1}{6}\int_{0}^{1}\left( tA+\left( 1-t\right) B\right)
^{r}dt-\int_{0}^{1}t\left( 1-t\right) \left( tA+\left( 1-t\right) B\right)
^{r}dt  \label{e.4.8} \\
& \leq \frac{1}{16}\left[ \frac{A^{r}+B^{r}}{2}-\left( \frac{A+B}{2}\right)
^{r}\right]   \notag
\end{align}%
if either $1\leq r\leq 2$ or $-1\leq r\leq 0$ and $A,$ $B>0.$

From (\ref{e.4.2.a}) we derive 
\begin{align}
0& \leq \frac{1}{6}\int_{0}^{1}\left( tA+\left( 1-t\right) B\right)
^{-1}dt-\int_{0}^{1}t\left( 1-t\right) \left( tA+\left( 1-t\right) B\right)
^{-1}dt  \label{e.4.9} \\
& \leq \frac{1}{64}\left[ A^{-1}\left( B-A\right) A^{-1}-B^{-1}\left(
B-A\right) B^{-1}\right]   \notag
\end{align}%
for $A,$ $B>0.$

From (\ref{e.4.3}) we obtain the logarithmic inequality 
\begin{align}
0& \leq \int_{0}^{1}t\left( 1-t\right) \ln \left( tA+\left( 1-t\right)
B\right) dt-\frac{1}{6}\int_{0}^{1}\ln \left( tA+\left( 1-t\right) B\right)
dt  \label{e.4.10} \\
& \leq \frac{1}{16}\left[ \ln \left( \frac{A+B}{2}\right) -\frac{\ln A+\ln B%
}{2}\right] ,  \notag
\end{align}%
while from (\ref{e.4.6}), the inequality 
\begin{align}
0& \leq \int_{0}^{1}t\left( 1-t\right) \ln \left( tA+\left( 1-t\right)
B\right) dt-\frac{1}{6}\int_{0}^{1}\ln \left( tA+\left( 1-t\right) B\right)
dt  \label{e.4.11} \\
& \leq \frac{1}{64}\left[ \int_{0}^{\infty }\left( s1_{H}+B\right)
^{-1}\left( B-A\right) \left( s1_{H}+B\right) ^{-1}ds\right.   \notag \\
& -\left. \int_{0}^{\infty }\left( s1_{H}+A\right) ^{-1}\left( B-A\right)
\left( s1_{H}+A\right) ^{-1}ds\right]   \notag
\end{align}%
for $A,$ $B>0.$

\end{document}